\documentclass[11pt]{amsart}
\usepackage{amsmath,amssymb,amsthm}
\usepackage{tikz,xstring,ifthen}
\usepackage[numbers,comma,square,sort&compress]{natbib}
\usepackage{mathabx}   
\usepackage{stackengine}
\usepackage{caption}
\usepackage{subcaption} 
\RequirePackage{color}
\RequirePackage[colorlinks,urlcolor=my-blue,linkcolor=my-red,citecolor=my-green]{hyperref}

\numberwithin{figure}{section}

\usepackage{enumerate}

\usepackage{scalerel}    
\usepackage{stmaryrd}   
\usepackage{mathabx}   

\definecolor{my-blue}{rgb}{0.0,0.0,0.6}
\definecolor{my-red}{rgb}{0.5,0.0,0.0}
\definecolor{my-green}{rgb}{0.0,0.5,0.0}
\definecolor{nicos-red}{rgb}{0.75,0.0,0.0}

\newtheorem{theorem}{\sc Theorem}[section]
\newtheorem{lemma}[theorem]{\sc Lemma}

\newtheorem{corollary}[theorem]{\sc Corollary}
\newtheorem{assumption}[theorem]{\sc Assumption}

\numberwithin{equation}{section}
\theoremstyle{remark}
\newtheorem{remark}[theorem]{Remark}
\newtheorem{example}[theorem]{Example}

\newcommand{\be}{\begin{equation}}
\newcommand{\ee}{\end{equation}}
\newcommand{\beq}{\begin{equation}}
\newcommand{\eeq}{\end{equation}}

\newcommand{\nn}{\nonumber}
\providecommand{\abs}[1]{\vert#1\vert}

\newcommand{\fl}[1]{\lfloor{#1}\rfloor}

\def\cM{\mathcal{M}}

\def\cS{\mathcal{S}} 

\def\cK{\mathcal{K}}

\def\bE{\mathbb{E}} 

\def\bP{\mathbb{P}}

\def\bR{\mathbb{R}}
\def\bZ{\mathbb{Z}}

\def\Z{\bZ}  \def\R{\bR}

\def\e{\varepsilon}
\def\ind{\mathbf{1}}

\def\m1{\mathbf{1}}

\def\Vvv{{\rm\mathbb{V}ar}}

  \def\wh{\widehat} \def\wb{\overline} 

\def\E{\bE}
\def\P{\bP} 

\def\funct lp{L} 
\def\funct lpbar{\bar L} 

\definecolor{darkgreen}{rgb}{0.0,0.5,0.0}
\definecolor{darkblue}{rgb}{0.0,0.0,0.3}
\definecolor{nicosred}{rgb}{0.65,0.1,0.1}
\definecolor{light-gray}{gray}{0.7}
\allowdisplaybreaks[1]

\usepackage{filemod}

\def\cif1{v}   

\def\deq{\overset{d}=}

\newcommand\bbullet{{\raisebox{0.5pt}{\scaleobj{0.6}{\bullet}}}} 

\def\tsp{\hspace{0.55pt}}  
\def\tspa{\hspace{0.7pt}}  
\def\tspb{\hspace{0.9pt}}

\usepackage[nomessages]{fp}

  \def\wh{\widehat} \def\wb{\overline}

 \def\rnder{f}   

\def\cv{{\sigma^2_*}} 
\def\cva{\sigma_*} 
\def\cm{c_M} 
\def\Cm{C_M} 
\def\frq{\theta_1} 
\def\cmu{D_\mu} 
\newcounter{usedm}
\newcounter{usedn}
\newcommand*{\lppwl}[6]{
	
	\FPeval{\m}{round(#5*(#3-#1)/(#3+#4-#1-#2),0)}
	
	\FPeval{\n}{round(#5*(#4-#2)/(#3+#4-#1-#2),0)}
	
	\FPeval{\stlength}{(#3+#4-#1-#2)/#5}
	
	\setcounter{usedm}{0}
	\setcounter{usedn}{0}
	
	\foreach\i in{1,...,{#5}}{
		\FPrandom{\x}
		
		\FPeval{\y}{(\m-\theusedm)/(\m+\n-\theusedm-\theusedn)}
		
		\FPeval{\startx}{#1+(\theusedm)*\stlength}
		\FPeval{\starty}{#2+(\theusedn)*\stlength}
		
		\ifthenelse{\lengthtest{\x pt > \y pt}}{
			\FPset{\endx}{\startx}
			\FPadd{\endy}{\starty}{\stlength}
			\stepcounter{usedn}
		}
		{
			\FPadd{\endx}{\startx}{\stlength}
			\FPset{\endy}{\starty}
			\stepcounter{usedm}
		}
		\draw[#6](\startx,\starty) -- (\endx,\endy);
	};
}

\begin{document}
	\title[Running maximum]{Bound on the running maximum of a random walk with small drift}

	\author[O.~Busani]{Ofer Busani}
	\address{Ofer Busani\\ University of Bristol\\  School of Mathematics\\ Fry Building\\ Woodland Rd.\\   Bristol BS8 1UG\\ UK.}
	\email{o.busani@bristol.ac.uk}
	\urladdr{https://people.maths.bris.ac.uk/~di18476/}
	\thanks{O. Busani was supported by EPSRC's EP/R021449/1 Standard Grant.} 
	
	\author[T.~Sepp\"al\"ainen]{Timo Sepp\"al\"ainen}
	\address{Timo Sepp\"al\"ainen\\ University of Wisconsin-Madison\\  Mathematics Department\\ Van Vleck Hall\\ 480 Lincoln Dr.\\   Madison WI 53706-1388\\ USA.}
	\email{seppalai@math.wisc.edu}
	\urladdr{http://www.math.wisc.edu/~seppalai}
	\thanks{T.\ Sepp\"al\"ainen was partially supported by  National Science Foundation grant DMS-1854619 and by the Wisconsin Alumni Research Foundation.}

	\date{\today}
	\begin{abstract} 
	We derive a lower bound for the probability that a random walk with i.i.d.\ increments and small negative drift $\mu$ exceeds  the value $x>0$ by time $N$. When the moment generating functions are bounded in an interval around the origin,   this probability can be bounded below by $1-O(x\abs\mu\log N)$.	The approach is elementary and does not use strong approximation theorems.  
	\end{abstract}
	\maketitle
	\tableofcontents
	\section{Introduction}
	
	\subsection{Background} 
This paper arose from the need of a random walk estimate for the authors' article \cite{busa-sepp-poly} on directed polymers.  This estimate is a {\it positive} lower bound on the running maximum of a random walk with a small {\it negative} drift.  Importantly, the bound had to come with sufficient control over its constants so that it would apply to an infinite sequence of random walks whose drift scales to zero as the maximum  is taken over expanding  time intervals. The natural approach via a Brownian motion embedding appeared to not give either the desired precision or the uniformity.  Hence we resorted to a proof from scratch.   For possible wider use we derive the result here under general hypotheses on the distribution of the step of the walk. 

The polymer application of the result	pertains to the exactly solvable log-gamma polymer on the plane.  The objective of \cite{busa-sepp-poly} is to prove that there are no bi-infinite polymer paths on the planar lattice $\Z^2$.  The technical result is that there does not exist any nontrivial  Gibbs measures on bi-infinite paths that satisfy the Dobrushin-Lanford-Ruelle (DLR) equations under the Gibbsian specification defined by the quenched polymer measures.  In  terms of limits of finite polymer distributions,  this means that as the southwest and northeast  endpoints of a random polymer path are taken to opposite infinities, the   middle portion of the path  escapes.  This is proved by showing that in the limit  the probability that the  path   crosses the $y$-axis along a given edge   decays to zero.  This probability in turn is controlled by considering stationary polymer processes from the two endpoints to an interval along the $y$-axis.  The crossing probability can be controlled in terms of a maximum of a random walk.  In the case of the log-gamma polymer,  the steps of this random walk are distributed as the difference of two independent log-gamma variables.   The case needed for \cite{busa-sepp-poly} is treated in Example \ref{ex:drw2} below.

	\subsection{The question considered} 
	   We seek a lower bound on the probability that the running maximum of a  random walk  with negative drift reaches a level $x>0$.  To set the stage, we discuss the matter through Brownian motion.   Let $S^N_n=\sum_{i=1}^n X^N_i$ be a   random walk with drift  $\E(X^N_1)=\mu_N=\mu N^{-1/2}<0$, and such that the random walks $S^N$ converge weakly to a  Brownian motion with drift $\mu<0$. The probability of the event
	\begin{align*}
		\sup_{1\leq n \leq N}S^N_n>x
	\end{align*}
	should be approximately the same  as that of
	\begin{align}\label{bwd}
		\sup_{0\leq s \le 1} (B_s+s\mu) > xN^{-1/2}. 
	\end{align} 
	This latter can   be computed   (see \eqref{kmt612}) to be
	\begin{align}\label{bc}
		\P\big\{\sup_{0\leq s \leq  1} (B_s+s\mu ) > xN^{-1/2}\tspa\big\}=1-O(\abs\mu xN^{-1/2}\tspa).
	\end{align}
	This suggests that we should aim for an estimate of the type 
	\begin{align}\label{rwc}
		\P\Big(\sup_{1\leq n \leq N}S^N_n>x\Big)\ge 1-O(\abs{\mu_N} x ).  
	\end{align}
	To reach this precision weak convergence is not powerful enough, for a weak approximation of random walk by Brownian motion reaches only a precision of $O(N^{-1/4})$ \cite{sawyer1972rates,fraser1973rate}.  
	Our estimate \eqref{mr} below does almost capture \eqref{rwc}: we have to allow   an additional $\log N$ factor inside the $O(\cdot)$ and consider $x$ of order at least $(\log N)^2$.  
	
	
	The by-now  classical  Koml\'os-Major-Tusn\'ady (KMT)  coupling \cite{koml-majo-tusn-76}  gives  a strong approximation of random walk with Brownian motion with a  discrepancy that grows  logarithmically in time.  This precision is sufficient for us, as we illustrate in Section \ref{sec:kmt}.   The problem is now the control of the constants in the approximation.  
	Uniformity of the constants is necessary for our application in \cite{busa-sepp-poly}.  But verifying this uniformity from the original work  \cite{koml-majo-tusn-76}  appeared to be a highly  nontrivial task. In the end it was more efficient to derive the estimate (Theorem \ref{thm:lm} below)  from the ground up. 
	
The difficulty of the original KMT proof has motivated several recent attempts at simplification and better understanding of the result, such as Bhattacharjee and Goldstein \cite{bhattacharjee2016strong}, 
	  Chatterjee \cite{chatterjee2012new}, and Krishnapur \cite{krishnapur2020one}.  There is another strong approximation result due to Sakhanenko \cite{sakhanenko1984rate} which, according to p.~232 of \cite{chatterjee2012new},  ``is so complex that some researchers are hesitant to use it''.

\subsection{Sketch of the proof} 
Our proof is elementary. The most sophisticated result used is the Berry-Esseen theorem.  
Given a random walk of small drift $\mu<0$, our approach  can be summarized in two main steps:
	\begin{enumerate}
		\item Up to the time the random walk hits the level $-\e \abs\mu^{-1}$ it behaves like an unbiased random walk.
		\item By the time the random walk hits the level $-\e\abs\mu^{-1}$ it will have had about $\log_2(\e|\mu|^{-1}x^{-1} )$ independent opportunities to hit the level $x$.  By the previous step this implies that the probability on the left-hand side of \eqref{rwc} is of order $1-(1/2)^{\log_2(\e|\mu|^{-1}x^{-1} )}=1-O(|\mu| \e^{-1} x )$.
	\end{enumerate}
As we will take $\e=(\log N)^{-1}$ in the proof, we will obtain the right order in \eqref{rwc} up to a logarithmic factor (Theorem \ref{thm:lm}).  After the statement of the theorem we illustrate it with examples.  Then we demonstrate that even if we knew that the constants in the KMT approximation can be taken uniform, the result would not be essentially stronger in the regime in which we apply our result. 


	\section{Main result} \label{sec:mr}
	For each $N\in\Z_{>0}$, let $\{X^N_i\}_{i\geq 1}$ be a sequence   of non-degenerate   i.i.d.\ random variables.
	Denote their moment generating function by 
	\begin{align*}
		M_N(\theta)&=\E\big(e^{\theta X^N_1}\big).
	\end{align*}
Write   $M_N^{(0)}=M_N$ and $M_N^{(i+1)}=(d/d\theta)M_N^{(i)}$.

	 \begin{assumption}\label{ass:rw}
	 	\begin{enumerate} [{\rm(i)}] \itemsep=3pt 
	 		We assume that the random variables $\{X^N_i\}_{i\geq 1}$ satisfy the following:
	 		\item   There exists an open interval $(-\theta_0, \theta_0)$ around the origin  on which each moment generating function $M_N$ is finite.  Furthermore, there exists a finite   constant  
	  $C_M$ and $\frq >0$  such that  we have the uniform  bounds 
	 		\begin{align}
	 		|M_N^{(i)}(\theta)|&\leq \Cm  \quad \text{for all $N$, $0\leq i\leq 3$, and $\theta\in[-\frq ,\frq ]$}
			\label{ed} 
			\end{align} 
for the compact interval $[-\frq ,\frq ]\subset(-\theta_0, \theta_0)$. 			
	 		
	 		\item There exists a finite constant $\cva >0$ such that 
	 		\begin{align}\label{cv}
	 			\E\big[(X^N_1)^2\big]\ge \sigma_N^2= \Vvv(X^N_1) \ge   \cv  \quad \text{for all $N$.} 
	 		\end{align}
	 		\item There exists a finite constant $\cmu>0$ such that  the expectations $\mu_N=\E\big(X^N_1\big)$
			   satisfy
	 		\begin{align}\label{r}
	 			-\cmu(\log N)^{-3} \leq \mu_N \le 0  \quad \text{for all $N$.}  
	 		\end{align}
	 	\end{enumerate}	 
	 \end{assumption}

	The conditions in Assumption \ref{ass:rw} are fairly natural.  Note that \eqref{ed} has to be checked only for $i=0$ at the expense of shrinking the interval $[-\frq, \frq]$ and increasing $\Cm$. 
To make a positive maximum possible, 
 condition \eqref{cv} ensures enough diffusivity  and condition \eqref{r}  limits the strength of the negative drift.  The bound  \eqref{mr} below shows that $\cmu$ has to be vanishingly small in order for the result to be nontrivial. 
		
	For $m\ge 1$ let $S^N_m=\sum_{k=1}^{m}X^N_k$  be   the random walk   associated with the steps  $\{X^N_i\}_{i
	\geq 1}$.   Here is the main theorem.

\begin{theorem}\label{thm:lm}
	There exist finite constants $C$ and $N_0$ that depend on $\frq ,\cv, \cmu $ and $\Cm  $ such that, for every $N\ge N_0$ and $x\ge(\log N)^2 $, 
	\begin{align}\label{mr}
	\mathbb{P}\Big(\max_{1\leq m \leq N}S^N_m\le x\Big)\leq C\tspa x(\log N)(\tspa |\mu_N|\vee N^{-1/2}\tspa ). 
	\end{align}
\end{theorem}

\medskip 

\begin{remark}[The constants in the theorem]   The constant $C$ in the upper bound \eqref{mr} is given by 
  \be\label{CC}  C=4\exp\bigl\{ 4(C_0+4c_\tau^{-1})(1+\cmu) + 8\frq^{-1}  \bigl(1+ \log(e^{\frq }\Cm +1)\bigr)     +4 \Cm \bigr\} + 3 \ee 
 where 
  \be\label{CC1}
	C_0=2(\cva^{-1}+9e^{\frq}\cva^{-3}\Cm^{\tspa 5/2}  ) + 2e^{8\Cm \cva^{-4}+8\cva^{-2} }c_\tau^{-1}+12\cva^{-2}\,, 
	\ee
	  \be\label{CC2}  c_\tau=\log\frac2{1+\Phi_{\cv } [-2,2]}\,, \ee
and 	$\Phi_{\cv }$ is the mean zero  Gaussian distribution with variance $\cv$. 
	
 
 Throughout the proof we state explicitly  
 the various conditions  $N\ge N_0$ required along the way.  Let us assume that $N\ge 2$ so that $\log N$  does not vanish.  Then all the  conditions $N\ge N_0$  can be combined into a single condition of the form  
\be\label{fN_0}   f(\Cm, \cmu,  \cv, \frq, N) \ge 1  \ee 
where the function 
  $f$ is a strictly positive continuous function on $\R_{>0}^4\times\R_{\ge2}$, nondecreasing in  $\frq$,   nonincreasing in $\Cm$ and $\cmu$, but depends on $\cv$ in both directions.  When $(\Cm, \cmu,  \cv, \frq)$  is restricted to a compact subset $\cK$ of $\R_{>0}^4$, there exists a finite index $N_\cK$ such that $f(\Cm, \cmu,  \cv, \frq, N)$ is a nondecreasing function of $N\ge N_\cK$ for any fixed $(\Cm, \cmu,  \cv, \frq)\in \cK$, and 
  \[  \lim_{N\to\infty}\; \inf_{(\Cm, \cmu,  \cv, \frq)\tsp \in\tsp \cK}  f(\Cm, \cmu,  \cv, \frq, N)  = \infty. \] 
  
  In particular, for each compact subset $\cK\subset\R_{>0}^4$ there exists a finite index $N_{0,\cK}$ such that  \eqref{fN_0} holds 
   for all $N\ge N_{0,\cK}$ and all $(\Cm, \cmu,  \cv, \frq)\in \cK$. 
   Furthermore, it is evident from \eqref{CC}-\eqref{CC2} that $C$ is a continuous function of $(\Cm, \cmu,  \cv, \frq)\in\R_{>0}^4$.   We conclude with the following local uniformity statement.

\end{remark} 

   \begin{corollary} \label{cor:lm} 
  For each compact subset $\cK\subset\R_{>0}^4$  there exist finite constants $C_{0,\cK}$ and  $N_{0,\cK}$ such that the following holds:   the estimate \eqref{mr} with $C=C_{0,\cK}$ on the right-hand side  is valid  whenever  $N\ge N_{0,\cK}$,  simultaneously  for all walks $\{S^N_m\}_{m\ge 1}$ that satisfy Assumption \ref{ass:rw}  with  parameters  $(\Cm, \cmu,  \cv, \frq)\in \cK$. 
 \end{corollary}

We illustrate the result with some examples. 	
	
	\begin{example}[Gaussian random walk] 
		Let $B_t$ be a Brownian motion,  $\mu<0$, and define  the random walk $S^N_m=B_{m}+mN^{-1/2}\mu$ . 
%
We can verify that the bound \eqref{mr} is off by a logarithmic factor in this case, by comparison with the running maximum of the Brownian motion. 
		For $x>0$ and large enough $N$
		\be\label{op}\begin{aligned}
		\P\big(\max_{1\le m\le N}S^N_m\leq x\big)&\geq \P\big(\sup_{0\le t\le N}B_t+tN^{-1/2}\mu\leq x\big)\\
		&\geq 1-e^{-2xN^{-1/2}|\mu|}\geq x|\mu_N|=xN^{-1/2}|\mu|.
		\end{aligned}\ee
		where the middle inequality follows from \eqref{kmt612} with $\mu(N)=\mu$ and $b(N)=xN^{-1/2}$.\eqref{op} shows that the optimal error is at most  $O(x|\mu_N|)$, and that Theorem \ref{thm:lm}, if not optimal, is only $\log N$ away from being so.

\end{example}

A natural way to produce examples is to take $X^N_i$ as the difference of two independent random variables whose means come closer as $N$ grows  and whose variances stay bounded and bounded away from zero. 

\begin{example}[Exponential walk]\label{ex:Exp}
Consider a random walk $S_n=\sum_{k=1}^n X_k$ with step distribution  $X_k\deq Y_\alpha-Y_\beta$ where $Y_\alpha$ and $Y_\beta$ are two independent exponential random variables with rates $\alpha$ and $\beta$, respectively.  $\mu=\E[X_k]=\frac{\beta-\alpha}{\alpha\beta}$ so we assume that $\alpha>\beta$.   The distribution of the  supremum of $S_n$ is  well-known and  also feasible  to compute (Example (b) in Section XII.5 of Feller \cite{fellII}): 
 for $x>0$,   
\begin{align*} 
 \P\bigl( \sup_{n\ge0} S_n \le x) =1-  \frac\beta\alpha e^{-(\alpha-\beta)x}
 =\beta \abs\mu  \bigl( 1+\beta x\bigr) + O(\mu^2x^2)
  \end{align*} 
  where we assume $\abs\mu x$ small and expand $e^s=1+s+O(s^2)$. 
  We obtain a lower bound:
  \begin{align*} 
 \P\bigl( \max_{1\le n\le N} S_n \le x) &\ge  P\bigl( \sup_{n\ge0} S_n \le x) 
  =\beta \abs\mu  \bigl( 1+\beta x\bigr) + O(\mu^2x^2) \\
  &\ge \beta^2 \abs\mu  x  + O(\mu^2x^2). 
  \end{align*} 
Thus  for $\abs\mu\ge N^{-1/2}$ and small $x\abs\mu$, the upper bound \eqref{mr} loses only a logarithmic factor.   
 
That $\max_{1\le n\le N} S_n$ is close to the overall maximum $\sup_{n\ge0} S_n$ in the case $\abs\mu\ge N^{-1/2}$ is a consequence of the fact that the overall maximum is taken at a random time of order $\mu^{-2}$.  
This claim  is seen conveniently through ladder intervals $\{T_i\}_{i\ge 1}$.  These are the intervals $T_i=\tau_i-\tau_{i-1}$ between  successive ladder epochs defined by $\tau_0=0$ and 
\[ \tau_i=\inf\{ n>\tau_{i-1}:  S_n > S_{\tau_{i-1}}\}. \]
The distribution of $T_i$ is given by 
\[   \P(T_i=\infty)=1-\frac\beta\alpha  \quad\text{and}\quad 
\P(T_i=n) = C_{n-1}\frac{\alpha^{n-1}\beta^n}{(\alpha+\beta)^{2n-1}}\quad \text{for } n\in\Z_{>0},  \]
where $\{C_k\}_{k\ge0}$ are the Catalan numbers. (This calculation can be found in Lemma B.3 in the appendix of \cite{fan-sepp-arxiv}.) 
Set $T_0=0$ and let $N=\max\{ n\ge0 : T_n<\infty\}$ be   the number   of finite  ladder intervals.  The maximum $\sup_{n\ge 0}S_n$ is taken at time $\zeta=\sum_{i=1}^N T_i$.  One  calculates $\E[\zeta]=\frac1{\alpha\beta}\mu^{-2}$ and 
$\Vvv[\zeta]=c_{\alpha, \beta} \mu^{-4}$. Thus for large enough $k$,  $\P(\zeta> k\mu^{-2}) \le C_{\alpha, \beta} k^{-2}$. 
\end{example} 

\begin{example}[Log-gamma walk]\label{ex:drw2}
This is the application of Theorem \ref{thm:lm} used in \cite{busa-sepp-poly}. 

	Let $G^\lambda$ denote generically a parameter $\lambda$ gamma random variable, that is,  $G^\lambda$ has density function $f(x)=\Gamma(\lambda)^{-1}x^{\lambda-1}e^{-x}$ on $\R_{>0}$. 
	For $\alpha, \beta>0$  let  $S^{\alpha,\beta}_m=\sum_{i=1}^m X^{\alpha,\beta}_i$  denote the random walk where the distribution of the   i.i.d.\ steps $\{X^{\alpha,\beta}_i\}_{i\geq 1}$ is specified by 
	\[	X^{\alpha,\beta}_1 \; \deq \; \log G^{\alpha}-\log G^{\beta} \]
	with two  independent gamma variables $G^{\alpha}$  and $G^{\beta}$ on the right. 
	
Let  $\psi_0(s)=\Gamma'(s)/\Gamma(s)$ be the digamma function and $\psi_1(s)=\psi_0'(s)$ the trigamma function on $\R_{>0}$. Their key properties are that $\psi_0$  is strictly increasing  with $\psi_0(0+)=-\infty$ and $\psi_0(\infty)=\infty$,  while $\psi_1$ is strictly decreasing and strictly convex with  $\psi_1(0+)=\infty$ and $\psi_1(\infty)=0$.

Fix a compact interval $[\rho_{\rm min}, \rho_{\rm max}]\subset(0,\infty)$.  Fix a positive constant $a_0$ and let $\{s_N\}_{N\ge 1}$ be a sequence of nonnegative reals such that $0\le s_N\le a_0(\log N)^{-3}$. 
Define a set of admissible pairs 
\[  \cS_N=\{(\alpha, \beta):  \alpha, \beta\in [\rho_{\rm min}, \rho_{\rm max}], \; 
-s_N\le \alpha-\beta \le 0\}. 
\] 
For $(\alpha, \beta)\in\cS_N$,  the mean step satisfies 
	\be\label{mu}\begin{aligned} 
	\mu_{\alpha,\beta}&=\E[X^{\alpha,\beta}_1] =\E[\log G^{\alpha}]- \E[\log G^{\beta} ]=\psi_0(\alpha)-\psi_0(\beta) \\
	&=\psi_1(\lambda)(\alpha-\beta) \;\in\; [-a_0\tspa \psi_1(\rho_{\rm min})(\log N)^{-3}, 0] 
	\end{aligned} \ee
where we used the mean value theorem with some $\lambda\in(\rho_{\rm min}, \rho_{\rm max})$.  We take $\cmu=a_0\tspa \psi_1(\rho_{\rm min})$. 

	The MGF of $X^{\alpha,\beta}_1$ is 
	\begin{align}\label{M}
	M_{\alpha,\beta}(\theta)&=\E\big[e^{\theta X^{\alpha,\beta}_1}\big]
	=\E\big[(G^{\alpha})^{\theta}\tspb\big]
	\, \E\big[(G^{\beta})^{-\theta}\tspb\big]
	=\frac{\Gamma(\alpha+\theta)\Gamma(\beta-\theta)}{\Gamma(\alpha)\Gamma(\beta)} 
	\end{align}
	for $\theta\in(-\alpha\tspb, \beta)$. 
For the interval in assumption \eqref{ed} we can take 
$[-\frq , \frq ]=[-\tfrac12\rho_{\rm min}, \tfrac12\rho_{\rm min}]$. Now  \eqref{ed} holds with a single constant $\Cm=\Cm(\rho_{\rm min}, \rho_{\rm max})  $  for all choices of $\alpha, \beta\in[\rho_{\rm min}, \rho_{\rm max}]$. 
	
The variance satisfies 
\[ \Vvv(X^{\alpha,\beta}_1) = \Vvv(\log G^{\alpha})+ \Vvv(\log G^{\beta} ) = \psi_1(\alpha)+ \psi_1(\beta)\ge 2\psi_1(\rho_{\rm max})=\cv.     \]

The constants $(\Cm, \cmu,  \cv, \frq)$ have been fixed and they   work simultaneously for all $(\alpha, \beta)\in\cS_N$  and all $N\ge 1$.  Define $C$ through \eqref{CC}--\eqref{CC2}. Choose $N_0$ so that \eqref{fN_0} holds for all $N\ge N_0$.   Now $C$ and $N_0$ are entirely determined by $(a_0, \rho_{\rm min}, \rho_{\rm max})$.   We state the result as a corollary of Theorem \ref{thm:lm}. 
	
 	\begin{corollary}  In the setting described above   the bound below holds for all  $N\ge N_0$,  $(\alpha, \beta)\in\cS_N$,  and $x\ge(\log N)^2$: 
		\begin{align*}
		\mathbb{P}\Big\{\tspa\max_{1\leq m \leq N}S^{\alpha,\beta}_m\le  x\Big\}\leq C\tsp x \tsp (\log N)(\mu_{\alpha,\beta} \vee N^{-1/2}\tspb).
		\end{align*}
	\end{corollary} 

 \end{example}

\medskip 

\section{Comparison with the KMT coupling} \label{sec:kmt}

As a counterpoint to our Theorem \ref{thm:lm}  we derive here an estimate for a single random walk with the  Koml\'os-Major-Tusn\'ady (KMT)  \cite{koml-majo-tusn-76}  coupling   with Brownian motion.  We emphasize though that {\it Theorem \ref{thm:kmt} below  is not an alternative to our Theorem \ref{thm:lm} because we do not know how the constants $C, K, \lambda$  below  depend on the distribution of the walk.  Hence without further work we cannot apply the resulting estimate \eqref{kmt620} to an infinite family of random walks.}  

However,  this section does  illustrate  that   in a certain regime of vanishing drift   the   estimates \eqref{mr} and \eqref{kmt620}  are essentially equivalent, as  explained below in Remark \ref{rm:kmt4}.  So even if one were to  conclude that the constants $C, K, \lambda$  below can be taken uniform, the result remains the same.  

Let $\wb S_n=\sum_{k=1}^n \wb X_k$ be a  mean-zero  random walk with i.i.d.\ steps $\{\wb X_k\}$ and unit variance $E[\,\wb X^{\tspa 2}\,]=1$. 
The  KMT coupling (Theorem 1 in \cite{koml-majo-tusn-76})  constructs this walk  together with a  standard Brownian motion $B_\bbullet$ on a probability space such that the following bound holds:  
\be\label{kmt5} 
P\bigl( \;\max_{1\le k\le N} \abs{\wb S_k-B_k} \ge C\log N +z\tspa \bigr) \le Ke^{-\lambda z}
\qquad \text{for all } N\in\Z_{>0} \ \text{ and }  \ z>0,  \ee
where $C, K, \lambda$ are finite positive constants determined by the distribution of $\wb X_k$.  

We apply this to the running maximum of a random walk with a negative drift. 

\begin{theorem}\label{thm:kmt}   Let 
$S_n=\sum_{i=1}^n X_i$ be a random walk with  i.i.d.\ steps $\{X_i\}$ that 
satisfy $E[e^{tX}]<\infty$ for $t\in(-\delta,\delta)$ for some $\delta>0$. Assume the drift is  negative:   $\mu=EX_1<0$,  and  the  variance $\sigma^2=E[\tspb(X_1-\mu)^2\tspb]>0$.  
Then there exists a constant $C_1$ determined by the distribution of the normalized variable  $\wb X_1=\sigma^{-1}(X_1-\mu)$ such that, for all real $x>0$ and integers $N>e^4$, 
\be\label{kmt620}  \begin{aligned}
P\bigl\{\tspb\max_{0\le k\le N} S_k < x \bigr\}
&\le  C_1\Bigl( \tspa N^{1-(\log N)/2} +   \frac{\sigma x+\sigma^2\log N}{N^{3/2}\mu^2} \tspb e^{(\sigma^{-1}x+\log N)\sigma^{-1}\mu} \tspa \Bigr)\\[3pt]
&\qquad 
+ 1 -  e^{2(\sigma^{-1}x+C_1\log N)\sigma^{-1}\mu} . 
\end{aligned}\ee 

\end{theorem} 

\medskip 

\begin{remark} \label{rm:kmt4}  To compare this estimate with Theorem \ref{thm:lm}, imagine that we can let  $\mu$ vary as a function of $N$ while preserving the constant $C_1$ in \eqref{kmt620}. 
 Consider the regime 
where $\sigma^2$ is constant,   $x>\log N$ and   $\abs\mu$ vanishes fast enough  so that $x\abs\mu$ stays bounded.   Then the first parenthetical expression  on the right of \eqref{kmt620}  is dominated by a constant multiple of $x  N^{-3/2}\mu^{-2}$.  To the last part  apply $1-e^{s}\le \abs s$  for $s<0$.  The bound \eqref{kmt620} becomes 
\be\label{kmt622}  
P\bigl\{\tspb\max_{0\le k\le N} S_k < x \bigr\}
\le  C_2 \tspb x  \bigl( 
N^{-3/2}\mu^{-2}   +   \abs\mu\bigr) . 
\ee 
The bound \eqref{mr} is worse than the one above  by at most a  $\log N$ factor,  and not at all if $\mu$ vanishes fast enough.   In particular, for the application in \cite{busa-sepp-poly}, the KMT bound cannot give anything substantially better than Theorem \ref{thm:lm}.  
\end{remark}

\medskip

\begin{proof}[Proof of Theorem \ref{thm:kmt}]  

Apply \eqref{kmt5} to 
the mean-zero unit-variance  normalized  walk  $\wb S_N= \sigma^{-1}(S_N-N\mu)$.  
To simplify some steps  below  we can  assume  that $C\ge 1\vee\lambda^{-1}$.
Let $x>0$ and $z=\lambda^{-1}\log N$. 
\begin{align}
\nn &P\bigl\{\tspb\max_{0\le k\le N} S_k < x\bigr\} = P\bigl\{ \tspb\max_{0\le k\le N} \bigl(\wb S_k  +k\sigma^{-1} \mu\bigr)  < \sigma^{-1} x\bigr\}\\[4pt] 
\label{kmt600} &\le  Ke^{-\lambda z}  +   P\bigl\{\tspb \max_{0\le k\le N} \bigl(B_k  +k\sigma^{-1}\mu\bigr)  < \sigma^{-1} x+ C\log N + z \bigr\}  . 
\end{align}

Let $M_k=\sup_{0\le s\le 1}( B_{k+s}-B_k)$.   Since  $\mu<0$, 
\begin{align*}
\sup_{0\le t\le N} \bigl(B_t  +t\sigma^{-1}\mu\bigr) 
&\le \max_{0\le k\le N} \bigl(B_k  +k\sigma^{-1}\mu\bigr)
+ \max_{0\le k\le N-1}  M_k. 
\end{align*}
With this we continue from above.  
\be\label{kmt603} \begin{aligned}
\text{line \eqref{kmt600}}  \; \le  \; 
 Ke^{-\lambda z} \; + \; & P\bigl\{\tspb \sup_{0\le t\le N} \bigl(B_t  +t\sigma^{-1}\mu\bigr) 
  < \sigma^{-1}x+ 2C\log N + z \bigr\}    \\[3pt] 
 &\qquad 
+  P\bigl\{ \tspb\max_{0\le k\le n-1}  M_k >  C\log N \bigr\} . 
\end{aligned}\ee 

We bound the two probabilities above separately. Recall   that $C\ge 1$.   For the running maximum of standard Brownian motion,  by (2.8.4) on page 96 of  \cite{kara-shre},   
\be\label{kmt609} \begin{aligned} 
P\bigl\{ \tspb\max_{0\le k\le N-1}  M_k >  \log N \bigr\} 
&\le N \tsp P\bigl\{ \tspa \sup_{0\le s\le 1}  B_s >  \log N \bigr\} 
= N \sqrt{2/\pi} \int_{\log N}^\infty e^{-y^2/2}\,dy  \\
&\le  \frac{N\sqrt{2/\pi}}{ \log N}   \int_{\log N}^\infty y\tspa e^{-y^2/2}\,dy
=  \frac{\sqrt{2/\pi}}{ \log N}  N^{1-(\log N)/2}   .   
\end{aligned} \ee

For the running maximum of  Brownian motion with drift,  use first Brownian scaling, and then the density of the hitting time $T_{b(N)}$ of the point 
$b(N)=N^{-1/2}(\sigma^{-1}x+ 2C\log N + z) $  with drift $\mu(N)=\sigma^{-1}N^{1/2}\mu<0$  
 from (3.5.12) on page 197 of  \cite{kara-shre}. 
\be\label{kmt612} \begin{aligned} 
& P\bigl\{\tspb \sup_{0\le t\le N} \bigl(B_t  +t\sigma^{-1}\mu\bigr) 
  < \sigma^{-1}x+ 2C\log N + z \bigr\}    \\[3pt] 
 &= P\bigl\{\tspb \sup_{0\le t\le 1} \bigl(B_t  +t\sigma^{-1}N^{1/2}\mu\bigr) 
  <  N^{-1/2}(\sigma^{-1}x+ 2C\log N + z)  \bigr\}    \\[3pt] 
  &= P\bigl\{\tspb \sup_{0\le t\le 1} \bigl(B_t  +t \mu(N)\bigr) 
  <  b(N)   \bigr\}  =  P^{(\mu(N))}\{  T_{b(N)} > 1 \}   \\[3pt] 
  &= {b(N)} \int_1^\infty \frac1{\sqrt{2\pi s^3}} e^{-(b(N)-\mu(N)s)^2/2s}\,ds
  +  P^{(\mu(N))}\{  T_{b(N)} = \infty \} \\[3pt] 
  &= {b(N)}  e^{b(N)\mu(N)}\int_1^\infty \frac1{\sqrt{2\pi s^3}} e^{-\tfrac12b(N)^2s^{-1}-\tfrac12\mu(N)^2s}\,ds 
  + 1-  e^{2b(N)\mu(N)} \\
 & \le 2 \tsp e^{b(N)\mu(N)} \tspb \frac{b(N)}{\mu(N)^2}  + 1-  e^{2b(N)\mu(N)} \\
 &  \le  2\tsp e^{(\sigma^{-1}x+\log N)\sigma^{-1}\mu} \tspb \frac{\sigma x+3C\sigma^2\log N}{N^{3/2}\mu^2}
 + 1 -  e^{2(\sigma^{-1}x+3C\log N)\sigma^{-1}\mu}. 
\end{aligned} \ee
The second last inequality dropped   the denominator $2\pi s^3\ge 1$ and the    term $-\tfrac12b(N)^2s^{-1}$  from the exponent, and then integrated. 
The last inequality substituted in 
$z=\lambda^{-1}\log N \le C\log N$ to bound 
\[   N^{-1/2}(\sigma^{-1}x+\log N) \le b(N)\le N^{-1/2}(\sigma^{-1}x+3C\log N). \]
The conclusion \eqref{kmt620} follows from  substituting into \eqref{kmt603}  the bounds from above.
\end{proof}

\medskip 

\section{Auxiliary facts}

Before starting the proof proper, we record some simple facts.  First, assumptions 
	 \eqref{ed} and  \eqref{cv} gives these  bounds:  
	\be\label{cubic}\begin{aligned}
		&0<\cv \leq \mu_{2,N}\equiv \E\big[(X^N_1)^2\big]=M_N^{(2)}(0)\leq C_M,\\[3pt] 
		&|\mu_{3,N}|\equiv |\E\big[(X^N_1)^3\big]| =\abs{M_N^{(3)}(0)} \le C_M, \\[3pt] 
		&\P(X^N_1>t)\leq C_Me^{-\frq t}. 
	\end{aligned}\ee

%
\begin{lemma}\label{lem:exp}   Let $\{Y_i\}$ be   i.i.d.\ random variables with common marginal distribution $\nu$.  Assume  that, for two constants $0<c_1,C_1<\infty$,  
\begin{align}
\E(e^{tY_1})\leq C_1 \quad\text{ for }\quad  t\in[0,c_1]\label{assu2}.
\end{align} 
Then 
	\begin{align*}
	\mu_{\rm max}=\mu_{\rm max}(\nu, n)\equiv\E\bigl[ \max\{0,Y_1,...,Y_n\}] \leq c_1^{-1}\log(C_1n+1). 
	\end{align*}
\end{lemma}
\begin{proof}
	For $0<t\leq c_1$, 
	\begin{align*}
	e^{t\mu_{\rm max}}\leq \E\big(e^{t(0\vee \max_{1\leq i  \leq n}Y_i)}\big)\leq 1+\E\Big(\sum_{i=1}^{n}e^{tY_i}\Big)=1+ n\E\big(e^{tY_1}\big)\leq C_1n+1,
	\end{align*}
	and the claim follows by taking $t=c_1$. 
	\end{proof}

%

Since $M_N''>0$ there is a unique minimizer 
\begin{align}\label{theta}
\theta^N_0 
=\arg \min\{ M_N(\theta)\}.
\end{align}
\begin{lemma}
Let $N_0$ be such that $\Cm\abs{\mu_N} \le \tfrac13\cv$ for $N\ge N_0$ and 
	set   $\cm =2\cva^{-2} $.  Then for $N\ge N_0$, 
	\begin{align*}
		0\le\theta^N_0\leq \cm |\mu_N|.
	\end{align*}
\end{lemma}


\begin{proof}   If $M'_N(0)=\mu_N=0$ then the minimum  is taken  at $\theta^N_0=0$. 

So suppose  $M'_N(0)=\mu_N<0$.   Expansion 
	for $\theta\in(0,\frq )$ gives, with some $\theta'\in(0,\theta)$,   
	\begin{align*}  
	M'_N(\theta)&=\mu_N+\mu_{2,N}\theta+  \tfrac12  M^{(3)}_N(\theta') \tspa \theta^2 
	\ge  \mu_N+\mu_{2,N}\theta -   \tfrac12 \Cm   \theta^2. 
	\end{align*}
Since $M'$ is strictly increasing and   $\cm =2/\cv\ge 2\mu_{2,N}^{-1}$,  by the choice of $N_0$ we have  for   $N\ge N_0$
	\begin{align*}
		& 
		M'_N(-\cm \mu_N) \ge M'_N\Big(-\frac{2\mu_N}{\mu_{2,N}}\Big)
		\ge -\mu_N -  2\Cm  \tspb \frac{\mu_N^2}{\mu_{2,N}^2} 
		\ge -\mu_N\bigl( 1 - \tfrac23 \cv /\mu_{2,N}^2\bigr) >0.   	
		\end{align*}
	It follows that there exists a unique $\theta^N_0\in (0,\cm  |\mu_N|)$ such that $M'_N(\theta^N_0)=0.$ 
	\end{proof}

Define a tilted measure $Q(d\omega)=\rnder^{\theta_0^N}_{N,n}(\omega)\P(d\omega)$ in terms of the  Radon-Nikodym derivative 
\begin{align*}
\rnder^{\theta_0^N}_{N,n}(\omega)=\frac{e^{\theta^N_0S^N_n}}{\E\big(e^{\theta^N_0S^N_n}\big)}.
\end{align*}
Denote the expectation under  $Q$ by $\E^Q$.   Increase $N_0$ further so that   $N\ge N_0$ implies   $\theta^N_0\in[-\frq /2,\frq /2]$ and $-\mu_N\le 2$.  Then  for $0\leq i \leq 3$ and   $\theta\in(-\frq /2,\frq /2)$, the MGF under $Q$ satisfies
\be\label{QM4}\begin{aligned}
	M^{(i)}_{Q,N}(\theta)&=\E^Q\big((X^N_1)^ie^{\theta X^1_N}\big)=M_N(\theta^N_0)^{-1}\E\big((X^N_1)^ie^{(\theta+\theta^N_0) X^1_N}\big) \\
	&=M_N(\theta^N_0)^{-1}M^{(i)}_N(\theta^N_0+\theta)\leq e^{-\mu_N \theta_0^N}\Cm  \leq e^{\frq }\Cm  ,
\end{aligned}\ee
where the first inequality  used Jensen's inequality and \eqref{ed}.
From this we get  moment bounds under  $Q$: for $0\leq i \leq 3$, 
\be\label{qm}\begin{aligned}
	\E^Q\big((X^N_1)^i\big)&=M^{(i)}_{Q,N}(0)\leq e^{\frq }\Cm. 
\end{aligned}\ee
For $\abs\theta\leq \theta_1$, there exists $\theta'\in(-\theta_1,\theta_1)$ 
\begin{align*}
	M_N^{(2)}(\theta)=\mu_{2,N}+M_N^{(3)}(\theta')\theta 
\end{align*}
Increase $N_0$ further if necessary so that 
  $\theta^N_0\le\frac{\cv}{2\Cm}$ for $N\ge N_0$ and we can write 
\begin{align*}
	M_N^{(2)}(\theta_0^N)\geq \cv-\Cm\theta_0^N\ge \frac{\cv}{2}. 
\end{align*}
Then from $\E^Q(X^N_1)=0$ and the third equation in \eqref{QM4}, 
\begin{align}\label{vq}
	\Vvv^Q(X^N_1)=\E^Q\big((X^N_1)^2\big)&=M^{(2)}_{Q,N}(0)=M_N(\theta^N_0)^{-1}M^{(2)}_N(\theta^N_0)\ge C_M^{-1}\frac{\cv}{2}.
\end{align}

\section{Proof of the main theorem}
To lighten the notation we omit the label $N$  from $\mu=\mu_N$ and $\theta_0=\theta^N_0$, and from some other notation that obviously depend on $N$. For $y>0$ let 
\begin{align*}
\tau_y=\inf\{m\ge 1: |S^N_m|\geq y\} 
\end{align*}
denote  the first hitting time of the cylinder of width $2y$.  Let $\Phi_{\sigma^2}$ denote the centered Gaussian distribution with variance $\sigma^2$. 

 
\begin{lemma}\label{lem:ubc1}   
	For real $k\ge 0$ and $y\ge y_0$  we  have  
	$ 
		\mathbb{P}(\tau_y> ky^2) \leq 2 e^{-c_\tau k},
	$ 
	where
\be\label{y_0}  y_0=1\,\vee\,\frac{6\Cm  \sigma_*^{-3}}{1-\Phi_{\cv }[-2,2]} 
\qquad\text{and}\qquad 
	c_\tau=\log\frac2{1+\Phi_{\cv } [-2,2]} \,\in\,(0,\log 2). 	
	\ee 
\end{lemma}
\begin{proof}
Let    $\widebar S^N_m=S^N_m-m\mu$ be   the centered walk.  Consider  an integer $k\ge 1$ and a real $y\ge 1$. 
Look at the process along time increments of size $\fl {y}^2$: 
	\begin{align*}
	\P(\tau_y>ky^2)&\leq \P(\tau_y>k\fl {y}^2)\leq\mathbb{P}( \,|S^N_{m\fl{y}^2}|\leq y \text{ for }  m=1,\dotsc,k\,)\\
	&\le \mathbb{P}( \,|S^N_{m\fl {y}^2}-S^N_{(m-1)\fl {y}^2}|\leq 2y \text{ for }  m=1,\dotsc,k\,)   \\
	&=\bigl(\mathbb{P}\{ \,S^N_{\fl {y}^2}\in [-2y,2y]\} \bigl)^k 
	= \bigl(\mathbb{P}\bigl\{ \,\fl {y}^{-1}\wb S^N_{\!\fl {y}^2}\in [-2-\mu \fl {y}\tspa,2-\mu \fl {y}]\bigr\} \bigl)^k  \\
	&\le \Bigl(\Phi_{\sigma_N^2} [-2-\mu \fl {y},2-\mu \fl {y}] + 3 \frac{\mu_{3,N}}{\sigma^{3}}\fl {y}^{-1} \Bigr)^k
	\le \bigl(\Phi_{\sigma_N^2} [-2,2] + 6\Cm  \sigma_*^{-3}y^{-1} \bigl)^k.
	\end{align*}	
	 The penultimate inequality is the Berry-Esseen Theorem. We use the version from \cite[Section 3.4.4]{durr} where the constant is 3. The last inequality is a simple property of a centered Gaussian.  Now for $y\ge y_0$
	 and  $c_\tau$ as above,
	 \begin{align*} 
	\sup_{\cv \leq \sigma^2\leq C_M}\Phi_{\sigma^2} [-2,2] + 3\Cm  \sigma_*^{-3} y^{-1} 
	& \le \Phi_{\cv } [-2,2] + 3\Cm  \sigma_*^{-3} y_0^{-1}\\
	&\le \tfrac{1}{2}(1+\Phi_{\cv } [-2,2])= e^{-c_\tau}.
	\end{align*} 
	 We have proved $ 
		\P(\tau_y>ky^2)\leq e^{-c_\tau k} $  for   $k\in\Z_{\ge0}$.
	Extend this to real   $k\in \R_{\ge0}$: 
	\begin{align*}
		\P(\tau_y>ky^2)&\leq \P(\tau_y>\fl {k}y^2)\leq e^{-c_\tau \fl{k}}\leq e^{-c_\tau( k-1)}=2(1+\Phi_{\cv } [-2,2])^{-1} e^{-c_\tau k}<2e^{-c_\tau k}.
	\qedhere\end{align*}
\end{proof}


Let $H_N=\mu^{-2}\wedge N$. By \eqref{r}
\begin{align}\label{HB}
	\cmu^{-2}(\log N)^6\leq H_N\leq N. 
\end{align}
   Define the truncated version of $\tau_y$ 
\begin{align*}
	\hat{\tau}_y=\tau_y\wedge H_N . 
\end{align*}
The following result shows that although the random walk $S^N_m$ has negative drift, up to times of order  $H_N$ it behaves similarly to an unbiased random walk in the following sense: if $y>0$ is not too small, but small compared to $H_N^{1/2}$, the probability that the random walk reaches level $y$ before   level $-y$ is close to $1/2$. Our choice of $H_N$ can be justified by decomposing the random walk into
\begin{align*}
	S^N_n=\sum_{i=1}^{n}\big(X_i^{N}-\mu\big)+n\mu.
\end{align*} 
For $\e>0$ small and $\abs\mu\ge N^{-1/2}$ (so that $H_N=\mu^{-2}$), 
\begin{align}\label{dt}
	(\e H_N)^{-1/2}S^N_{\e H_N}=(\e H_N)^{-1/2}\sum_{i=1}^{\e H_N}\big(X_i^{N}-\mu\big)+\e^{1/2}  .
\end{align} 
As
\begin{align*}
	(\e H_N)^{-1/2}\sum_{i=1}^{\e H_N}\big(X_i^{N}-\mu\big)\overset{d}\approx N(0,\sigma),
\end{align*}
we see that the left hand side of \eqref{dt} is dominated by the first term on the right hand side.  That is, up to time $\e H_N$ the random walk $S^N$ behaves approximately  like an unbiased random walk.


\begin{lemma}\label{lem:ep}  Let $y_0$ be as in \eqref{y_0}.  
	There exist  finite constants $N_0$ and $C_0$ such that, for $N\ge N_0$ and $y_0\le y\le (\log N)^{-1}H_N^{1/2}$, 
	\begin{align*}
	\P(S_{\hat{\tau}_y}\geq y)\geq \tfrac1{2}\Bigl[1-  C_0 H_N^{-\frac12}\big(y+(\log H_N)^2\big)-\frac2{\frq y}\log(e^{\frq }\Cm  H_N+1)
	\Bigr].   
	\end{align*}
	$C_0$ depends on $\frq ,\cv $ and $\Cm  $  while  $N_0$  depends on $\frq ,\cv $, $\cmu$  and $\Cm  $.  
\end{lemma}
\begin{proof}
 The constant $C_0$ comes as follows in terms of the constants previously  introduced above and new constants $C_2, C_3, C_4$   introduced below in the course of the proof:  
		\be\label{C_0} \begin{aligned}
	C_0=C_2+C_4 &=2(\cva^{-1}+9e^{\frq}\cva^{-3}\Cm^{5/2}  ) + 2C_3c_\tau^{-1}+6\cm \\
	&=2(\cva^{-1}+9e^{\frq}\cva^{-3}\Cm^{\tspa 5/2}  ) + 2e^{2\Cm  \cm ^2+4\cm }c_\tau^{-1}+6\cm .
	\end{aligned}\ee

	Under the measure $Q$, $S_n$ is a mean-zero random walk and hence a martingale.   Furthermore,  ${\hat{\tau}_y}$ is a bounded stopping time. From this, 
	\begin{align*}
	0= \int S_{\hat{\tau}_y} dQ=\int_{S_{\hat{\tau}_y}\geq y} S_{\hat{\tau}_y} dQ+\int_{S_{\hat{\tau}_y}\leq-y} S_{\hat{\tau}_y} dQ+\int_{S_{\hat{\tau}_y}\in (-y,y)} S_{\hat{\tau}_y} dQ.  
	\end{align*}
	On the event $S_{\hat{\tau}_y}\geq y$, we have $\hat{\tau}_y=\tau_y$ and $S_{\hat{\tau}_y-1}<y\le S_{\hat{\tau}_y}=S_{\hat{\tau}_y-1}+X^N_{\hat{\tau}_y}\le y+X^N_{\hat{\tau}_y}$ and so 
	\begin{align*}
	\int_{S_{\hat{\tau}_y}\geq y} S_{\hat{\tau}_y} \,dQ
	&\leq \int_{S_{\hat{\tau}_y}\geq y}\big(y+X^N_{\hat{\tau}_y}\big)\,dQ  
	\le \int_{S_{\hat{\tau}_y}\geq y}\big(y+0\vee\max_{1\leq i\leq H_N} X^N_i\big)\,dQ  \\[4pt] 
	&\le y\tspb Q(S_{\hat{\tau}_y}\geq y)+\mu_{\rm max}(Q, H_N)
	\;\leq \; y\tspb Q(S_{\hat{\tau}_y}\geq y)+ 2\frq ^{-1}\log(e^{\frq }\Cm  H_N+1),
	\end{align*}
	where we applied Lemma \ref{lem:exp}  under the distribution $Q$ with $C_1=e^{\frq }\Cm  ,c_1=\tfrac12\frq $ from \eqref{QM4}.  
Combine the displays above to obtain 
\begin{align*}
	 Q(S_{\hat{\tau}_y}\geq y) &\geq  -\, y^{-1}\!\!\!\int\limits_{S_{\hat{\tau}_y}\leq-y} S_{\hat{\tau}_y} dQ \; - \; y^{-1} \!\!\!\!\!\!\int\limits_{S_{\hat{\tau}_y}\in (-y,y)} S_{\hat{\tau}_y} dQ  \; - \; 2\frq ^{-1}y^{-1}\log(e^{\frq }\Cm  H_N+1) \\[4pt] 
	 &\ge  Q(S_{\hat{\tau}_y}\leq-y)-   Q(S_{\hat{\tau}_y}\in(-y,y))  -  2\frq ^{-1}y^{-1}\log(e^{\frq }\Cm  H_N+1)	. 
		\end{align*}
Use 
\begin{align*}
	Q(S_{\hat{\tau}_y}\leq-y)=1-Q(S_{\hat{\tau}_y}\geq y)-Q(S_{\hat{\tau}_y}\in(-y,y))
	\end{align*}
to rewrite the above as 
\be\label{o:600}  
	\tspb Q(S_{\hat{\tau}_y}\geq y)\geq \tfrac12 [1-2\tspb Q(S_{\hat{\tau}_y}\in(-y,y))- 2\frq ^{-1}y^{-1}\log(e^{\frq }\Cm  H_N+1)].
	\ee 

%
%

It remains to bound the probability on the right. 
  $S_{\hat{\tau}_y}\in(-y,y)$ forces $\hat{\tau}_y=H_N$ and thereby another application of   the Berry-Esseen theorem, while using \eqref{qm},  \eqref{vq} and $y\ge y_0\ge 1$, gives 
	\begin{align*}
	Q\bigl\{S_{\hat{\tau}_y}\in(-y,y)\bigr\}&=Q\bigl\{H_N^{-1/2}S_{H_N}\in(-H_N^{-1/2}y,H_N^{-1/2}y)\bigr\}\\
	& \le  \Phi_{\cv }(-H_N^{-1/2}y,H_N^{-1/2}y) + 3\frac{e^{\frq}\Cm  }{2^{-3/2}\Cm^{-3/2}\cva^3} H_N^{-\frac12} \\
	& \le  2(2\pi\cv)^{-1/2}yH_N^{-1/2} + 9\frac{e^{\frq}\Cm  }{\Cm^{-3/2}\cva^3}  H_N^{-\frac12}  \leq (\cva^{-1}+9e^{\frq}\cva^{-3}\Cm^{5/2}  )\tsp y \tsp H_N^{-\frac12}\\
	& \equiv \tfrac12C_2 \tsp y \tsp H_N^{-\frac12}. 
	\end{align*}

	Rewrite \eqref{o:600}  as  
	\begin{align}\label{qlb}
	Q(S_{\hat{\tau}_y}\geq y)\geq \tfrac12[1-C_2 \tsp y \tsp H_N^{-\frac12}-2y^{-1}\frq ^{-1}\log(e^{\frq }\Cm  H_N+1)].
	\end{align}
It remains to switch from $Q$ back to the original distribution $\P$. 	
Recall the Radon-Nikodym derivative  $\rnder^{\theta}_n={M(\theta)^{-n}}{e^{\theta S_n}}$.  Introduce a temporary quantity  $G_0>1$ to be chosen precisely below.  Decompose according to the value of $\hat{\tau}_y$ 
and use Cauchy-Schwarz:  
	\begin{align}
	Q(S_{\hat{\tau}_y}\geq y)
	&=\E\big[  \rnder^{\theta_0}_{ {\hat{\tau}_y}}  
	(\ind_{S_{\hat{\tau}_y}\geq y,\tspb {\hat{\tau}_y}\leq G_0}+\ind_{S_{\hat{\tau}_y}\geq y,\tspb {\hat{\tau}_y}> G_0})\big] \nn \\
	\label{plb}
	 &\leq \E\big[\rnder^{\theta_0}_{ {\hat{\tau}_y}} \tspa \ind_{S_{\hat{\tau}_y}\geq y,\tspb {\hat{\tau}_y}\leq G_0}\big]  
	+\Big(\E\big[(\rnder^{\theta_0}_{ {\hat{\tau}_y}})^2\big]\Big)^\frac12\Big(\P\{S_{\hat{\tau}_y}\geq y,\tspb {\hat{\tau}_y}> G_0\}\Big)^{\frac12}.   
	\end{align}
	
		Let us first bound the second term on line  \eqref{plb}.  
	Note that  $\rnder^{\theta}_n$ is a $\P$-martingale and ${\hat{\tau}_y}$ is a stopping time bounded by $H_N$. Hence  $(\rnder^{\theta}_n)^2$ is a submartingale and we have 
	\begin{align*}
	&\Big(\E[(\rnder^{\theta_0}_{ \hat{\tau}_y})^2\tspa]\Big)^\frac12\Big(\P\{S_{\hat{\tau}_y}\geq y,{\hat{\tau}_y}> G_0\}\Big)^{\frac12}\\
	&\leq\Big(\E[(\rnder^{\theta_0}_{H_N})^2\tspa]\Big)^\frac12\Big(\P\{S_{\hat{\tau}_y}\geq y,{\hat{\tau}_y}> G_0\}\Big)^{\frac12}
	=\bigg(\frac{M(2\theta_0)}{M(\theta_0)^{2}}\bigg)^{H_N/2}\Big(\P\{S_{\hat{\tau}_y}\geq y,{\hat{\tau}_y}> G_0\}\Big)^{\frac12}. 
	\end{align*}	
	  
To bound the $M$-factor on the right, 
expand $M$ and use  \eqref{ed},  \eqref{cubic} and  $\mu<0$. In the numerator, for some $\eta\in(0,2\theta_0)$, 
\[  M(2\theta_0) = 1+\mu2\theta_0+2\mu_2\theta_0^2+\tfrac86M^{(3)}(\eta)\theta_0^3
\le 1 +2\mu_2\theta_0^2+\tfrac43\Cm\theta_0^3
\]  
and similarly in the denominator: 
	\be\label{o:256} \begin{aligned}
	&\Big[M(2\theta_0)M(\theta_0)^{-2}\Big]^{H_N/2}\le \Big(1 +2\mu_2\theta_0^2+\tfrac43\Cm\theta_0^3\Big)^{H_N/2}\Big(1+\mu\theta_0+\tfrac12\mu_2\theta_0^2-\tfrac16\Cm\theta_0^3\Big)^{-H_N}\\
	&\leq \Big(1+2 \Cm  \cm ^2\mu^2+\tfrac43 C_M\cm^3\abs\mu^3\Big)^{\tfrac12\mu^{-2}} \Big(1- \cm \mu^2-C_M\cm^3\abs\mu^3\Big)^{-\mu^{-2}} \\
	&\leq e^{2\Cm  \cm ^2+4\cm }
	\equiv C_3. 
	\end{aligned}\ee
	Above we  used $ H_N\leq \mu^{-2}$ and increased $N_0$ once more so that $N\ge N_0$ guarantees  $\tfrac23\cm\abs\mu\le 1$,  $ \cm \mu^2+ C_M\cm^3\abs\mu^3\le \tfrac12$ and  $C_M\cm^3\abs\mu\le 1$. Then we applied the bounds 
\begin{align*} \Big(1+2 \Cm  \cm ^2\mu^2+\tfrac43 C_M\cm^3\abs\mu^3\Big)^{\tfrac12\mu^{-2}} &\le e^{\Cm\cm^2(1+\frac23\cm\abs\mu)} \le e^{2\Cm  \cm ^2},\\
 	\Big(1-  \cm \mu^2- C_M\cm^3\abs\mu^3\Big)^{-\mu^{-2}} &\le \Big(1+2\cm \mu^2\bigl(1+C_M\cm^3\abs\mu\bigr)  \Big)^{\mu^{-2}}\le e^{2\cm (1+C_M\cm^3\abs\mu) }\le e^{4\cm }, 
	\end{align*}
	where the second line also used   $(1-a)^{-1}\le 1+2a$ for $a\in[0,\tfrac12]$.  
	Put \eqref{o:256} back up, set $G_0=yH_N^{1/2}$, and apply  Lemma \ref{lem:ubc1} (for which we use the assumption  $y\ge y_0$):  
	\begin{align}\label{qub2}
	\Big(\E[(\rnder^{\theta_0}_{ \hat{\tau}_y})^2\tspa]\Big)^\frac12\Big(\P\{S_{\hat{\tau}_y}\geq y,{\hat{\tau}_y}> G_0\}\Big)^{\frac12}\leq C_3\bigl(\P\{{\hat{\tau}_y}>G_0\}\bigr)^\frac12 \leq 2 C_3 e^{-c_\tau H_N^{1/2}y^{-1}}.
	\end{align}

	Next we bound the first term on line  \eqref{plb}.   Use $M(\theta_0)\le 1$.  Let $\cM_n=\max_{1\leq i  \leq n}X^N_i$.
	\be\label{qub2.4} \begin{aligned}
		 &\E\big[\rnder^{\theta_0}_{ {\hat{\tau}_y}}\tspa\ind_{S_{\hat{\tau}_y}\geq y,\tspb {\hat{\tau}_y}\leq G_0}\big] =  
		 \E\Big[\frac{e^{\theta_0S_{\hat{\tau}_y}}}{M(\theta_0)^{\hat{\tau}_y}}\ind_{S_{\hat{\tau}_y}\geq y,\tspb {\hat{\tau}_y}\leq G_0}\Big] \\
		 &
		 \leq \E\Big[\frac{e^{\theta_0 S_{\hat{\tau}_y}}}{M(\theta_0)^{G_0}}\ind_{S_{\hat{\tau}_y}\geq y,\tspb \cM_{H_N}\leq(\log H_N)^2}\Big]
		 +
		 \E\Big[\frac{e^{\theta_0 S_{\hat{\tau}_y}}}{M(\theta_0)^{\hat{\tau}_y}}\ind_{S_{\hat{\tau}_y}\geq y,\tspb \cM_{H_N}>(\log H_N)^2}\Big]\\
		 &\leq \E\Big[\frac{e^{\theta_0(y+\cM_{H_N})}}{M(\theta_0)^{G_0}}\ind_{S_{\hat{\tau}_y}\geq y,\tspb \cM_{H_N}\leq (\log H_N)^2}\Big]
		 +\E\Big[\frac{e^{\theta_0 S_{\hat{\tau}_y}}}{M(\theta_0)^{\hat{\tau}_y}}\ind_{S_{\hat{\tau}_y}\geq y,\tspb \cM_{H_N}>(\log H_N)^2}\Big]
	\end{aligned}\ee
	Let us first bound the second term. Using Cauchy-Schwarz, the bound
	\begin{align*}
		\P(\cM_{H_N}>t)\leq H_NC_Me^{-\frq t},
	\end{align*} 
	 the bound \eqref{o:256}, and the tail bound in \eqref{cubic}, it follows that
	\begin{align*}
		\E\Big[\frac{e^{\theta_0 S_{\hat{\tau}_y}}}{M(\theta_0)^{\hat{\tau}_y}}\ind_{S_{\hat{\tau}_y}\geq y,\tspb \cM_{H_N}>(\log H_N)^2}\Big]
		&\leq\Big(\E[(\rnder^{\theta_0}_{H_N})^2\tspa]\Big)^\frac12\Big(\P\{\cM_{H_N}>(\log H_N)^2\}\Big)^{\frac12}
		\\
		&\leq C_3H_N^{1/2}C^{1/2}_Me^{-\frac12\frq (\log H_N)^2}.
	\end{align*}
		The first term on the last  line of \eqref{qub2.4}  is bounded as follows, with $G_0=yH_N^{1/2}$. 
	\begin{align*}
		&\E\Big[\frac{e^{\theta_0(y+\cM_{H_N})}}{M(\theta_0)^{G_0}}\ind_{S_{\hat{\tau}_y}\geq y,\tspb \cM_{H_N}\leq (\log H_N)^2}\Big]\leq \E\Big[\frac{e^{\theta_0(y+(\log H_N)^2)}}{M(\theta_0)^{G_0}}\ind_{S_{\hat{\tau}_y}\geq y}\Big]\\
		&\leq \P(S_{\hat{\tau}_y}\geq y) {e^{\cm H_N^{-1/2}[y+(\log H_N)^2]}}M(\theta_0)^{ -\tspa  yH_N^{1/2}}\leq \P(S_{\hat{\tau}_y}\geq y) {e^{\cm H_N^{-1/2}[y+(\log H_N)^2]}}e^{\cm yH_N^{-1/2}}\\
		&=\P(S_{\hat{\tau}_y}\geq y) {e^{\cm H_N^{-1/2}[2y+(\log H_N)^2]}}\\
		&\leq \P(S_{\hat{\tau}_y}\geq y)\big[1+2\cm H_N^{-1/2}[2y+(\log H_N)^2]\big]
		\leq \P(S_{\hat{\tau}_y}\geq y)+2\cm H_N^{-1/2}[2y+(\log H_N)^2]. 
	\end{align*}
	 We used above Jensen's inequality in the form  $M(\theta_0)^{-yH_N^{1/2}}   \le e^{-\theta_0\mu yH_N^{1/2}}$,  the definition of $H_N$ in the form  $\abs\mu H_N^{1/2}\le 1$, 
	 and then $\theta_0\leq \cm \abs\mu\leq \cm H^{-1/2}$. 
	 Furthermore, by \eqref{HB} and our assumption $y\le (\log N)^{-1}H_N^{1/2}$ we have 
\[  \cm H_N^{-1/2}[2y+(\log H_N)^2]\le \cm (2+\cmu) (\log N)^{-1} \le \log 2\]  
 where we choose $N_0$ large enough so that the last inequality holds for $N\ge N_0$. Then we applied  the inequality $e^x\le 1+2x$ for $x\in[0,\log 2]$.  
 
 Going back to \eqref{qub2.4}, 
	  for  $N\ge N_0$, 
	\begin{align*}
		E\big[\rnder^{\theta_0}_{ {\hat{\tau}_y}}\tspa\ind_{S_{\hat{\tau}_y}\geq y,\tspb {\hat{\tau}_y}\leq G_0}\big] 
		&\leq 	\P(S_{\hat{\tau}_y}\geq y)+2\cm H_N^{-1/2}[2y+(\log H_N)^2]+C_3H_N^{1/2}C^{1/2}_Me^{-\frac12\frq(\log H_N)^2}\\
		&\leq 	\P(S_{\hat{\tau}_y}\geq y)+3\cm H_N^{-1/2}[2y+(\log H_N)^2].
	\end{align*}
The second inequality is guaranteed for example by choosing $N_0$ large enough so that $N\ge N_0$ implies 
\[ \cmu^{-2}(\log N)^6\ge \tspb e^{\frq^{-1}} \qquad\text{and}\qquad 
\cm C_3^{-1} \Cm^{-1/2}\bigl(\log [\cmu^{-2} (\log N)^6\tspa]\bigr)^2\ge e^{\frac12\frq^{-1}}.\]
This works  due to the lower bound \eqref{HB} on $H_N$ and because the function $f(x)=xe^{-\frac12\frq(\log x)^2}$ achieves its maximum $e^{\frac12\frq^{-1}}$ at $x=e^{\frq^{-1}}$ after which it decreases.

	Combine the above with  \eqref{qub2} on line \eqref{plb} to  get this upper bound:
		\be\label{dlb}\begin{aligned}
	Q(S_{\hat{\tau}_y}\geq y)
	&\leq \P(S_{\hat{\tau}_y}\geq y)+2C_3 e^{-c_\tau H_N^{1/2}y^{-1}}+3\cm H_N^{-1/2}[2y+(\log H_N)^2]\\
	&\le \P(S_{\hat{\tau}_y}\geq y)+2C_3 c_\tau^{-1} H_N^{-1/2}y+3\cm H_N^{-1/2}[2y+(\log H_N)^2]\\
	&\le \P(S_{\hat{\tau}_y}\geq y)+ C_4 \tsp H_N^{-1/2}[  y+(\log H_N)^2]  
	\end{aligned}\ee
	where  $C_4=2C_3c_\tau^{-1}+6\cm$. 
The second inequality above came from   $xe^{-x}\leq e^{-1}$  for $x\ge 0$.   
	Put  \eqref{dlb} and  \eqref{qlb} together to obtain the claim of the lemma. 
\end{proof}

By adjusting a constant we can replace $\hat{\tau}_y$ with $\tau_y$ in the previous estimate.


\begin{corollary}\label{cor:St}  Under the assumptions of Lemma \ref{lem:ep}, with $C_{10}=C_0+2c_\tau^{-1}$, 
	\begin{align}\label{qub4}
		\P(S_{\tau_y}\geq y)\geq 
	\tfrac1{2}\Bigl[1-  C_{10} H_N^{-\frac12}\big(y+(\log H_N)^2\big)-\frac2{\frq y}\log(e^{\frq }\Cm  H_N+1)
	\Bigr]
	\end{align}
\end{corollary}
\begin{proof}   The assumption $y_0\le y\le H_N^{1/2}$ implies that  Lemma \ref{lem:ubc1}  applies to give  
\be\label{o:290}   \P(\tau_y>H_N)   \le e^{-c_\tau H_Ny^{-2}} \le e^{-c_\tau H_N^{1/2}y^{-1}}
\le c_\tau^{-1} H_N^{-1/2}y.\ee
 The claim then  comes from Lemma \ref{lem:ep} and 
	$
		\P(S_{\tau_y}\geq y)\geq \P(S_{\hat{\tau}_y}\geq y)-\P(\tau_y>H_N).  
	$
\end{proof}


For $w>0$   truncate:  
\begin{align*}
\wh X^{N,w}_i=X^{N}_i\ind_{\{X^N_i\geq -w\}}-w\ind_{\{X^N_i< -w\}}
\qquad\text{and}\qquad 
\wh S^{N,w}_n=\sum_{i=1}^{n}\wh X^{N,w}_i.
\end{align*}
Define
\begin{align*}
t_y=\inf\{m\ge 1:|\wh S^{N,w}_m|\geq y\}.
\end{align*}
We  transfer bound \eqref{qub4} to the truncated walk $\widehat S$. The reason is that the proof of the forthcoming Lemma \ref{lem:ub} is easier  for the truncated RW.  

\begin{corollary}\label{cor:plb}  Under the assumptions of Lemma \ref{lem:ep}, with $C_{11}=C_0+4c_\tau^{-1}$, 
	\begin{align*}
	\P(\wh S^{N,w}_{t_y}\geq y)\geq  	\tfrac1{2}\Bigl[1- C_{11} H_N^{-\frac12}\big(y+(\log H_N)^2\big)-\frac2{\frq y}\log(e^{\frq }\Cm  H_N+1)-H_NC_Me^{-\frq w}
	\Bigr]. 
	\end{align*}
	\begin{proof}
		Note that 
		\begin{align}\label{qub3}
		&\P\big(\wh S^{N,w}_m\neq S^N_m \  \text{ for some $1\leq m\leq H_N$}\big)= \P\big(\wh X^{N,w}_i\neq X^N_i \quad \text{for some $1\leq i\leq H_N$}\big) \\
		&=\P\big( \inf_{1\leq i\leq H_N}X^N_i<-w\big)\leq H_NC_Me^{-\frq w}.\nonumber
		\end{align}
		 Moreover, 
		\begin{align}\label{qub5}
			\P(\wh S^{N,w}_{t_y}\geq y)
			&\geq \P(S^N_{\tau_y}\geq y,\,\tau_y\leq H_N, \,\wh S^{N,w}_m= S^N_m \quad \text{for all $1\leq m\leq H_N$})\nonumber\\
			&\geq\P(S^N_{\tau_y}\geq y)-\P(\tau_y> H_N)-\P(\wh S^{N,w}_m\ne  S^N_m \quad \text{for some $1\leq m\leq H_N$})\nonumber\\
			&\geq\P(S^N_{\tau_y}\geq y)- c_\tau^{-1} H_N^{-1/2}y-H_NC_Me^{-\frq w},
		\end{align}
		where we used \eqref{qub3} and \eqref{o:290}.
		Combine the above with  \eqref{qub4}  to obtain the result.
	\end{proof}
\end{corollary}
We turn to the main argument of the proof of Theorem \ref{thm:lm}, that is, to show that the probability of the random walk $\wh{S}_m$ to hit the level $x$ before hitting the level $-\e H^{1/2}$ is close to $x|\mu|$.  This gives  rise to  the  error term in \eqref{mr}.   We sketch the reasoning.   

  Let us try to hit the level $x>0$ starting from the origin. By Corollary \ref{cor:plb} there is a probability $\approx1/2$ to hit  $x$ before  hitting $-x$. Suppose we failed and hit $-x$ first. We have another  chance to hit $x$ by going $2x$ upward from the level $-x$. By Corollary \ref{cor:plb} the probability of going $2x$ up to the level $x$ before going $2x$ down to the level $-4x$ is $\approx 1/2$. We continue this way until we either hit the level $x$ or the level $-\e  H_N^{1/2}$. How many trials to hit $x$  do we have before we hit $-\e  H_N^{1/2}$? Approximately $K=\log_2(x^{-1}\e  H_N^{1/2})$. The trials are independent  and so the probability of hitting the level $-\e  H_N^{1/2}$ before hitting the level $x$ is $\approx 2^{-K}=Cx|\mu|$, which is what we seek.  
  
  We introduce the notation to  make the  sketch precise.  See Figure \ref{fig: points} for an illustration. 
  

Define  $K=\fl{\log_2 ({x}^{-1}(\log N)^{-1}  H_N^{1/2})}-2$. 
For $i\ge 0$ set $L_i= 2^{i+2}-3$. Inductively these satisfy  $L_0=1$ and $L_i=2L_{i-1}+3$.  Furthermore, 
\begin{align*}
	xL_K\leq (\log N)^{-1}   H_N^{1/2}.
\end{align*}
Define the stopping times
\begin{align*}
T_0&=\inf\{n:|\wh S^{N,x}_n|\geq xL_0\}\\
\text{and } \quad 
T_i&=\inf\{n\ge T_{i-1}:\wh S^{N,x}_n\leq -xL_i \text{ or }\wh S^{N,x}_n\geq x \}. 
\end{align*}
Note that $T_i=T_{i-1}$ is possible.  

\begin{lemma}\label{lem:ub}
	There exist finite constants $C_{12}$ and $N_0$ such that  for  $N\ge N_0$ and $x\ge(\log N)^2$,  
	\begin{align*}
	\P\bigl(\,\max_{1 \leq m\leq T_K} \wh S^{N,x}_m< x\bigr)\leq  C_{12}\tspb x\tspb (\log N)H_N^{-1/2},
	\end{align*}
	where
	\begin{align*}
		C_{12}=4\exp\bigl\{ 4(C_0+4c_\tau^{-1})(1+\cmu) + 8\frq^{-1}  \bigl(1+ \log(e^{\frq }\Cm +1)\bigr)     +4 \Cm \bigr\} 
	\end{align*}
	and $C_0$ in the expression above  is from \eqref{C_0}. 
\end{lemma}
	\begin{proof}
Since  $C_0\ge 2$, we have $C_{12}\ge 4e^8\ge 2^{10}$. Then we can assume that $x\le 2^{-10}  (\log N)^{-1}H_N^{1/2}$, for otherwise the bound on the probability is $> 1$. This   guarantees that $K\ge 8$.  It also implies that unless    $\abs\mu\le 2^{-10}(\log N)^{-3}$, the result is trivial. 

Since $\wh X^{N,x}_i\ge-x$, 
 \be\label{coa} \{\wh S^{N,x}_{T_i}\leq-xL_i\}=\{-x(L_i+1)<\wh S^{N,x}_{T_i}\leq-xL_i\}\ee
  and 
\begin{align*}
	\big\{\wh S^{N,x}_{T_i}\leq-xL_i,\wh S^{N,x}_{T_{i+1}}\leq-xL_{i+1}\big\}\subseteq \{T_i<T_{i+1}\}.  
\end{align*}	
		Note that
		\begin{align}\label{e}
			E\equiv \Big\{\max_{1 \leq m\leq T_K} \wh S^{N,x}_m< x\Big\}\subseteq \bigcap_{1\leq i \leq K} \{\wh S^{N,x}_{T_i}\leq -x L_i\}. 
		\end{align}
		Due to \eqref{coa}
		\be\label{sm}\begin{aligned}
			&\P\big(\wh S^{N,x}_{T_0}\leq -xL_0,...,\wh S^{N,x}_{T_{i-1}}\leq -xL_{i-1},\wh S^{N,x}_{T_i}\leq -xL_i\big)\\
			&=\P\big(\wh S^{N,x}_{T_i}\leq -xL_i|\wh S^{N,x}_{T_0}\leq -xL_0,...,\wh S^{N,x}_{T_{i-1}}\leq -xL_{i-1}\big) \\
			&\qquad\qquad \cdot \; \P\big(\wh S^{N,x}_{T_0}\leq -xL_0,...,\wh S^{N,x}_{T_{i-1}}\leq -xL_{i-1}\big)\\
			&\leq\P\big(\wh S^{N,x}_{T_i}\leq -xL_i|\wh S^{N,x}_{T_0}\leq -xL_0,...,\wh S^{N,x}_{T_{i-1}}=-x(L_{i-1}+1)\big)\\
			&\qquad\qquad \cdot \;\P\big(\wh S^{N,x}_{T_0}\leq -xL_0,...,\wh S^{N,x}_{T_{i-1}}\leq -xL_{i-1}\big)\\
			&=\P\big(\wh S^{N,x}_{t_{x(L_{i-1}+2)}}\leq -x(L_{i-1}+2)\big)\P\big(\wh S^{N,x}_{T_0}\leq -xL_0,...,\wh S^{N,x}_{T_{i-1}}\leq -xL_{i-1}\big). 
		\end{aligned}\ee
		 The last equality   used the definition of the stopping time $t_y$, the definition of $L_i$, and the  Markov property.  For $1\leq i \leq K$ define the events 
		\begin{align*}
		A^N_i=\{\wh S^{N,x}_{t_{x(L_{i-1}+2)}}\leq -x(L_{i-1}+2)\}.
		\end{align*}
		Applying  \eqref{sm} to \eqref{e} repeatedly, 
		\begin{align*}
			\P(E)\leq \P\Big(\bigcap_{1\leq i \leq K} \{\wh S^{N,x}_{T_i}\leq -xL_i\}\Big)\leq \prod_{1\leq i \leq K}\P(A^N_i).
		\end{align*}
		Let $x\ge (\log N)^2$.   Recall  that by \eqref{HB}, $(\log H_N)^2H_N^{-1/2}\leq \cmu(\log N)^{-1}$ and $H_N\le N$.  Apply Corollary \ref{cor:plb} with $w=x$ and $y_i=x(L_{i-1}+2) \in [x, (\log N)^{-1}  H_N^{1/2}]$  for $i=1,\dotsc,K$ and  $N\ge N_0$ to get this estimate: 
			\begin{align*}
			\P(A^N_i)&\leq \tfrac1{2}\Bigl[1+C_{11} H_N^{-\frac12}
			\big(y_i+(\log H_N)^2\big)
			+ \frac2{\frq y_i}\log(e^{\frq }\Cm  H_N+1) 
			+H_NC_Me^{-\frq x}
			\Bigr]\\
			&\leq \tfrac1{2}\Bigl[1+C_{11}(1+\cmu)  (\log N)^{-1}   
			+ \frac{2\log(e^{\frq }\Cm  N+1)}{\frq (\log N)^2} 
			+C_M N^{1-\frq (\log N)^2}
			\Bigr]\\
			&\leq  \tfrac1{2}\bigl[1+C_A(\log N)^{-1}],
		\end{align*}
		where we set 
		\begin{align*}
			C_A=C_{11}(1+\cmu) + 2\frq^{-1}  \bigl(1+ \log(e^{\frq }\Cm +1)\bigr)     + \Cm 
		\end{align*}
and if necessary we increase $N_0$ further so that $N^{1-\frq (\log N)^2}\le (\log N)^{-1}$ for $N\ge N_0$.  
 	
	
Continue with the above estimate, 
		\begin{align*}
			\P(E)&\leq  \prod_{i=1}^K\P(A_i^N)\leq  \Big(\tfrac1{2}\bigl[1+C_A(\log N)^{-1}]\Big)^K\\
			&=\Big(\tfrac1{2}[1+C_A(\log N)^{-1}]\Big)^{\fl{\log_2 ({x}^{-1}(\log N)^{-1}  H_N^{1/2})}-2}\\
			&\leq 4x(\log N)  H_N^{-1/2} \tsp[1+C_A(\log N)^{-1}]^{\log_2\! N}\\
			&\leq 4e^{4C_{\!A}}x(\log N)  H_N^{-1/2}=4e^{4C_{\!A}}x(\log N)(|\mu|\vee N^{-1/2}),
		\end{align*}
where we used $\log_2\!N=\frac{\log N}{\log 2}\leq 4\log N$.		
	\end{proof}

We are ready to prove Theorem \ref{thm:lm}. By Lemma \ref{lem:ub},   by the time $\wh S$ hits the level $(\log N)^{-1}  H_N^{1/2}$, with high probability  it has hit level $x$ as well.   It remains  to verify  the two points below. 
\begin{enumerate} [(i)] 
	\item $\wh S$ is close to $S$ on the time interval $[1,N]$. This follows from a union bound and the  exponential tail of $X^{N}_1$.
	\item With high probability by time $N$ we hit the boundary of the cylinder of width  $(\log N)^{-1}  H_N^{1/2}$. This follows from Lemma \ref{lem:ubc1}.
\end{enumerate}

\begin{figure}[t]
	\centering%
	\begin{tikzpicture}[scale=0.8, every node/.style={transform shape}]
	\draw  (0,1) -- (10,1);
	\node [scale=1][left] at (0,1) {$x$};
	\draw  (0,0) -- (10,0);
	\node [scale=1][left] at (0,0) {$0$};
	\draw  (0,-1) -- (10,-1);
	\node [scale=1,red][left] at (0,-1) {$-xL_0=-x$};
	\draw  (0,-2) -- (10,-2);
	\node [scale=1][left] at (0,-2) {$-2x$};
	\draw  (0,0) -- (10,0);
	\draw  (0,-3) -- (10,-3);
	\node [scale=1][left] at (0,-3) {$-3x$};
	\draw  (0,-4) -- (10,-4);
	\node [scale=1][left] at (0,-4) {$-4x$};
	\draw  (0,-5) -- (10,-5);
	\node [scale=1,red][left] at (0,-5) {$-xL_1=-5x$};
	\node [scale=1][] at (0,-6) {$\vdots$};
	\draw [red] (0,-7) -- (10,-7);
	\node [scale=1,red][left] at (0,-7) {$-xL_K=(\log N)^{-1}  H_N^{1/2}$};
	\draw [dashed,red] (0,0) -- (2,1);
	\draw  (0,0) -- (2,-1.5);
	\draw [dashed,red] (2,-1.5) -- (6.5,1);
	\node [scale=1,][left] at (1.8,-1.6) {$\wh S^N_{T_0}$};
	\draw  (2,-1.5) -- (5.5,-5.3);
	\node [scale=1,][left] at (5.3,-5.5) {$\wh S^N_{T_1}$};
	\end{tikzpicture}
	\caption{\small By the time the random walk $\wh S^N$ exits the cylinder of radius $(\log N)^{-1}  H_N^{1/2}$ it has had about $K$ independent opportunities to hit the level $x$, each with probability close to $1/2$.}\label{fig: points}
\end{figure}
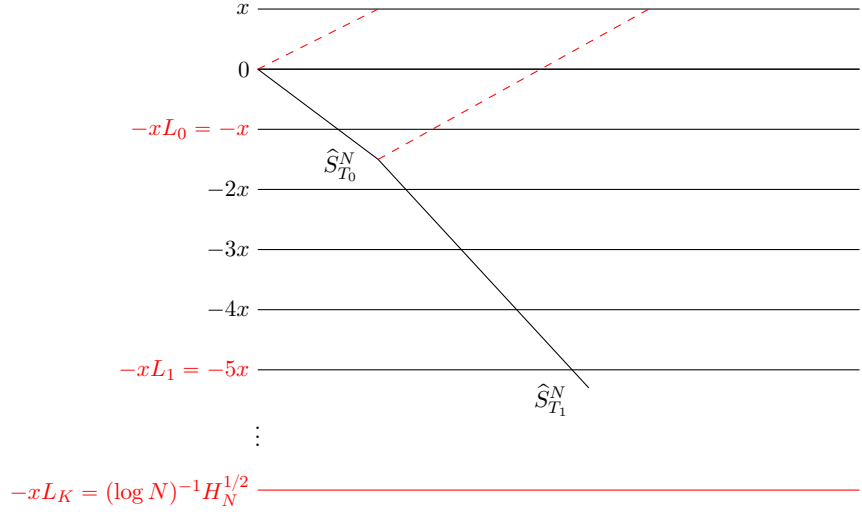	
\begin{proof}[Proof of Theorem \ref{thm:lm}]
	Consider $x\ge (\log N)^2$.  Observe that
	\begin{align*}
		\Big\{\max_{1\leq m \leq N}|\wh S^{N,x}_m|\geq (\log N)^{-1}  H_N^{1/2},\max_{1 \leq m\leq T_K} \wh S^{N,x}_m> x\Big \}\subseteq \Big\{\max_{1\leq m \leq N}\wh S^{N,x}_m\geq x\Big\}.
	\end{align*}
	Indeed, on the event $\max_{1\leq m \leq N}|\wh S^{N,x}_m|\geq (\log N)^{-1}  H_N^{1/2} {\ge xL_K}$ we have $T_K\leq N$.
	
	Next,  
	\begin{align}\label{cS}
	\P\big(\wh S^{N,x}_m\neq S^N_m \quad \text{for some $1\leq m\leq N$}\big)
	&=\P\big(X^N_i< -x \quad \text{for some $1\leq i\leq N$}\big)\\
	&\leq  \Cm Ne^{-\frq x} \leq \Cm Ne^{-\frq (\log N)^2}\nonumber.
	\end{align}
	By Lemma \ref{lem:ubc1}, \eqref{cS} and  Lemma \ref{lem:ub}, 
		\begin{align*}
		&\P\bigl(\,\max_{1\leq m \leq N}\wh S^{N,x}_m\geq x\bigr)
		\geq 1-\P\bigl(\,\max_{1\leq m \leq N}|\wh S^{N,x}_m|<(\log N)^{-1}  H_N^{1/2}\bigr)-\P\bigl(\,\max_{1 \leq m\leq T_K} \wh S^{N,x}_m\leq  x\bigr)\\
		&\ge 1-\Big[\P\bigl(\,\max_{1\leq m \leq N}| S^{N}_m|<(\log N)^{-1}  H_N^{1/2}\bigr)+\P\big(\wh S^{N,x}_m\neq S^N_m \quad \text{for some $1\leq m\leq N$}\big)\Big]\\
		&\qquad\qquad 
		-\P\bigl(\,\max_{1 \leq m\leq T_K} \wh S^{N,x}_m\leq  x\bigr)\\
		&=1-\Big[\P\bigl(\tau_{(\log N)^{-1}  H_N^{1/2}}>N\bigr)+\P\big(\wh S^{N,x}_m\neq S^N_m \quad \text{for some $1\leq m\leq N$}\big)\Big]\\
		&\qquad\qquad -\P\bigl(\,\max_{1 \leq m\leq T_K} \wh S^{N,x}_m<  x\bigr)\\
		&\geq 1-\big[2 e^{-c_\tau N(\log N)^2 H_N^{-1}}+\Cm Ne^{-\frq (\log N)^2}\big]-C_{12}\tsp x\tsp(\log N)  H_N^{-1/2}\\
		&\geq 1-2 e^{-c_\tau (\log N)^2 }-\Cm Ne^{-\frq (\log N)^2}-C_{12}\tsp x\tsp(\log N)  H_N^{-1/2}\\
		&\ge 1-  (C_{12}+2)x(\log N)  H_N^{-1/2} . 
 	\end{align*}
	To get the inequalities above for $N\ge N_0$ 
we increase  $N_0$ if necessary  so that $N\ge N_0$ guarantees $(\log N)^{-1}  H_N^{1/2}\ge y_0$ to apply Lemma \ref{lem:ubc1}, and furthermore so that 
$2e^{-c_\tau (\log N)^2 } \vee \Cm Ne^{-\frq (\log N)^2} \le (\log N)^3 N^{-1/2}$ to get the last inequality.  	
	

Now the final inequality: 
	\begin{align*}
		\P\bigl(\max_{1\leq m \leq N}S^N_m\geq x\bigr)
		&\geq \P\bigl(\max_{1\leq m \leq N}\wh S^{N,x}_m\geq x\bigr) -\P\bigl(\wh S^{N,x}_m\neq S^N_m \quad \text{for some $1\leq m\leq N$}\bigr)\nonumber\\
		&\geq 1-(C_{12}+2)x(\log N)  H_N^{-1/2} -\Cm Ne^{-\frq (\log N)^2} \nonumber\\
		&\geq 1-(C_{12}+3)x(\log N)(|\mu|\vee N^{-1/2}). 
	\end{align*}
	Theorem \ref{thm:lm} has been proved. 
\end{proof}
\small

	\bibliography{Timo_old_bib}
\bibliographystyle{plain}
\end{document}